\documentclass{siamltex1213}

\usepackage[latin1]{inputenc}    % LaTeX, comprends les accents !
\usepackage[T1]{fontenc}      % Police contenant les caractÃ¨res franÃ§ais
\usepackage{geometry}         % DÃ©finir les marges
\usepackage[english]{babel}  % Placez ici une liste de langues, la
\usepackage{amsfonts,amssymb,amsmath}
\usepackage{geometry}
\usepackage{mathrsfs}
\usepackage{varioref}
\usepackage{dsfont}
\usepackage{graphicx}
\usepackage{color}

\newtheorem{rema}[theorem]{Remark}
\newtheorem{hypo}{Hypothesis}

\DeclareMathOperator{\tr}{Tr}
\newcommand{\E}{\mathbb{E}}

\newcommand{\der}{\mathrm{d}}
\newcommand{\R}{\mathbb{R}}

\newcommand{\Pb}{\mathbb{P}}
\newcommand{\N}{\mathbb{N}}
\newcommand{\Q}{\mathbb{Q}}
\newcommand{\Pt}{\mathscr{P}_t}

\newcommand{\LL}{\mathscr{L}}

\newcommand{\F}{\mathscr{F}}
\newcommand{\Ga}{\mathscr{G}^1}
\newcommand{\LpC}{L_\mathscr{P}^p(\Omega, \mathscr{C}([0,T];H))}

\numberwithin{equation}{section}

\begin{document}

\title{A probabilistic approach to large time behaviour of mild solutions of HJB equations in infinite dimension}

\author{Ying Hu\thanks{IRMAR, Université Rennes 1, Campus de Beaulieu, 35042 Rennes Cedex, France (ying.hu@univ-rennes1.fr), partially supported by Lebesgue center of mathematics ("Investissements d'avenir"
program - ANR-11-LABX-0020-01)} \and Pierre-Yves Madec\thanks{IRMAR, Université Rennes 1, Campus de Beaulieu, 35042 Rennes Cedex, France (pierre-yves.madec@univ-rennes1.fr), partially supported by Lebesgue center of mathematics ("Investissements d'avenir"
program - ANR-11-LABX-0020-01)} \and Adrien Richou\thanks{Univ. Bordeaux, IMB, UMR 5251, F-33400 Talence, France.
(adrien.richou@math.u-bordeaux1.fr)}}

\maketitle

\begin{abstract}
We study the large time behaviour of mild solutions of HJB equations in infinite dimension by a purely probabilistic approach. For that purpose, we show that the solution of a BSDE in finite horizon $T$ taken at initial time behaves like a linear term in $T$ shifted with the solution of the associated EBSDE taken at initial time. Moreover we give an explicit rate of convergence, which seems to be new up to our best knowledge.
\end{abstract}

\section{Introduction}
We are concerned with the large time behaviour of solutions of the Cauchy problem in an infinite dimensional real Hilbert space $H$: 
\begin{align}\label{HJB cauchy changement de temps}
\left\{ 
\begin{array}{ll}
\frac{\partial u(t,x)}{\partial t} = \LL u(t,x) + f(x,\nabla u(t,x)G) , & \forall (t,x) \in \R_+ \times H,\\
u(0,x) = g(x),&\forall x \in H,
\end{array}
\right.
\end{align}
where $u : \R_+ \times H \rightarrow \R$ is the unknown function and $\LL$ is the formal generator of the Kolmogorov semigroup $\Pt$ of an $H$-valued random process solution of the following Ornstein-Uhlenbeck stochastic differential equation:
\begin{align*}
\left\{ 
\begin{array}{ll}
\der X_t = (AX_t + F(X_t^x))   \der t + G \der W_t, & t \in \R_+, \\
X_0 = x , & x \in H,
\end{array}
\right.
\end{align*}
with $W$ a Wiener process with values in another real Hilbert space $\Xi$, assumed to be separable, and $G$ a linear operator from $\Xi$ to $H$.
We recall that (formally), $\forall h : H \rightarrow \R$,
\begin{align*}
(\LL h)(x) = \frac{1}{2}\tr (GG^* \nabla^2 h(x)) + \langle Ax + F(x), \nabla h(x) \rangle.
\end{align*}
Our method uses only probabilistic arguments and can be described as follows.

First, let $(v,\lambda)$ be the solution of the ergodic PDE:
\begin{align*}
\LL v + f(x,\nabla v(x) G) -\lambda = 0, ~~~\forall x \in H.
\end{align*}
Then we have the following probabilistic representation. Let $(Y^{T,x},Z^{T,x})$ be solution of the BSDE:
\begin{align*}
\left\{
\begin{array}{l}
\der Y_s^{T,x} = -f(X_s^x,Z_s^{T,x}) \der s + Z_s^{T,x} \der W_s,\\
Y_T^{T,x} = g(X_T^x),
\end{array}
\right.
\end{align*}
and $(Y,Z,\lambda)$ be solution of the EBSDE:
\begin{align*}
\der Y_s = -(f(X_s^x,Z_s^x)-\lambda) \der s + Z_s^x \der W_s.
\end{align*}
Then we get
\begin{align*}
\left\{
\begin{array}{l}
Y_s^{T,x} = u(T-s,X_s^x),\\
Y_s^x = v(X_s^x).
\end{array}
\right.
\end{align*}
Finally, due to Girsanov transformations and the use of an important coupling estimate result, we deduce that there exists a constant $L \in \R$ such that for all $x \in H$,
\begin{align*}
Y_0^{T,x} - \lambda T - Y_0^x \underset{T \rightarrow + \infty}{\longrightarrow} L,
\end{align*}
i.e.
\begin{align*}
u(T,x) - \lambda T - v(x) \underset{T \rightarrow + \infty}{\longrightarrow} L.
\end{align*}
Our method uses not only purely probabilistic arguments, but also gives a rate of convergence:
\begin{align*}
|u(T,x) - \lambda T - v(x) - L| \leq C(1 + |x|^{2+\mu})e^{-\hat{\eta} T}.
\end{align*}
The constant $\mu$ appearing above is the polynomial growth power of $g(\cdot)$ and $f(\cdot,0)$ while $\hat{\eta}$ is linked to the dissipative constant of $A$.

Large time behaviour of solutions has been studied for various types of HJB equations of second order; see, e.g.,  \cite{BARLES_SOUGANIDIS_SPACE_TIME}, \cite{FUJITA_ISHII_LORETI_ASYMPTOTIC}, \cite{ICHIHARA_SHEU_HAMILTON_JACOBI_BELLMAN} and \cite{NAMAH_ROQUEJOFFRE_CONVERGENCE}. In \cite{BARLES_SOUGANIDIS_SPACE_TIME}, a result in finite dimension is stated under periodic assumptions for $f$ and a periodic and Lipschitz assumption for $g$. Furthermore, they assume that $f(x,z)$ is of linear growth in $z$ and bounded in $x$. In \cite{FUJITA_ISHII_LORETI_ASYMPTOTIC}, some results are stated in finite dimensionnal framework, under locally Hölder conditions for the coeffcients. More precisely, they assume that $f(x,z) = H_1(z) - H_2(x)$ with $H_1$ a Lipschitz function and with locally H\"older conditions for $H_2$ and $g$. They also treat the case of $H_1$ locally Lipschitz but consequently need to assume that $H_2$ and $g$ are Lipschitz. Furthermore, they only treat the Laplacian case, namely they assume that $G = 
I_d$. No result on rate of convergence is given in that paper. In \cite{ICHIHARA_SHEU_HAMILTON_JACOBI_BELLMAN}, the authors deal with the problem in finite dimension. They also only treat the Laplacian case and assume that $f(x,z)$ is a convex function of quadratic growth in $z$ and of polynomial growth in $x$.  No result on rate of convergence is given in this paper. Up to our best knowledge, the explicit rate of convergence only appears in Theorem $1.2$ of \cite{NAMAH_ROQUEJOFFRE_CONVERGENCE} but in finite dimension and under periodic assumptions for $f(\cdot,z)$ and $g(\cdot)$. Furthermore, they only deal with the Laplacian case and they assume restrictive assumptions on $f$ (i.e., there exists $0 < m < M$ such that $m < f(x,z) \leq M(1+|z|)$ and  boundedness hypotheses about the partial derivatives of first and second order of $f$).

In this paper, we will assume that $A$ is a dissipative operator, $G : \Xi \rightarrow H$ is an invertible and bounded operator, $g : H \rightarrow \R$ continuous with polynomial growth and $f : H \times \Xi^* \rightarrow \R$ continuous, with polynomial growth in the first variable and Lipschitz in the second variable.

The paper is organised as follows: In section $2$, we introduce some notations. In section $3$, we recall some results about existence and uniqueness results for solutions of an  Ornstein-Ulhenbeck SDE, a general BSDE and an EBSDE that will be useful for what follow in the paper. In section $4$, we study the behaviour of the solution of the BSDE taken at initial time when the horizon $T$ of the BSDE increases. More precisely, in a first time we are concerned with the path dependent framework, where a very general result can be stated. Then in the Markovian framework we obtain a more precise result for the behaviour of solutions and a rate of convergence is given.  In section $5$, we apply our result to an optimal control problem.

%COMMENTAIRE CADRE MARKOVIEN
%

\section{Notations}
We introduce some notations. Let $E_1$, $E_2$ and $E_3$ be real separable Hilbert spaces. The norm and the scalar product will be denoted by $|\cdot |$, $\langle\cdot,\cdot \rangle$, with subscripts if needed. $L(E_1,E_3)$ is the space of linear bounded operators $E_1 \rightarrow E_3$, with the operator norm, which is denoted by $|\cdot|_{L(E_1,E_3)}$ . The domain of a linear (unbounded) operator $A$ is denoted by $D(A)$. $L_2(E_1,E_3)$ denotes the space of Hilbert-Schmidt operators from $E_1$ to $E_3$, endowed with the Hilbert-Schmidt norm, which is denoted by $|\cdot|_{L_2(E_1,E_3)}$.

Given $\phi \in B_b(E_1)$, the space of bounded and measurable functions $\phi : E_1 \rightarrow \R$, we denote by $||\phi ||_0 = \sup_{x \in E_1} |\phi(x)|$. %If in addition, $\phi$ is also Lipschitz continuous then $||\phi ||_{\text{lip}} = ||\phi ||_0 + \sup_{x,x' \in E, x \neq x'}|\phi(x)-\phi(x')||x-x'|^{-1}$.

We say that a function $F : E_1 \rightarrow E_3$ belongs to the class $\mathscr{G}^1(E_1,E_3)$ if it is continuous, has a Gâteaux derivative $\nabla F(x) \in L(E_1,E_3)$ at any point $x \in E_1$, and for every $k \in E_1$, the mapping $x \mapsto \nabla F(x) k$ is continuous from $E_1$ to $E_3$. Similarly, we say that a function $F : E_1 \times E_2 \rightarrow E_3$ belongs to the class $\mathscr{G}^{1,0}(E_1 \times E_2 , E_3)$ if it is continuous, Gâteaux differentiable with respect to the first variable on $E_1 \times E_2$ and $\nabla_x F : E_1 \times E_2 \rightarrow L(E_1, E_3)$  is strongly continuous. In connection with stochastic equations, the space $\Ga$ has been introduced in \cite{FUHRMAN_TESSITORE_NONLINEAR_KOLMOGOROV_INFINITE_DIMENSION}, to which we refer the reader for further properties.

Given a real and separable Hilbert space K and a probability space $(\Omega, \F, \Pb)$ with a filtration $\F_t$, we consider the following classes of stochastic processes.
\smallskip \\
1. $L_\mathscr{P}^p(\Omega, \mathscr{C}([0,T];K))$, $p \in [1,\infty)$, $T > 0$, is the space of predictable processes $Y$ with continuous paths on $[0,T]$ such that 
\begin{align*}
|Y|_{L_\mathscr{P}^p(\Omega, \mathscr{C}([0,T];K))} = \E \sup_{t \in [0,T]} |Y_t|_K^p < \infty.
\end{align*}
\smallskip \\
2. $L_\mathscr{P}^p(\Omega, {L}^2([0,T];K))$, $p \in [1,\infty)$, $T > 0$, is the space of predictable processes $Y$ on $[0,T]$ such that 
\begin{align*}
|Y|_{L_\mathscr{P}^p(\Omega, L^2([0,T];K))}  = \E \left( \int_0^T |Y_t|_K^2 \der t \right)^{p/2} < \infty.
\end{align*}
\smallskip \\
3. $L_{\mathscr{P},\text{loc}}^2(\Omega, {L}^2([0,\infty);K))$ is the space of predictable processes $Y$ on $[0,\infty)$ which belong to the space $L_\mathscr{P}^2(\Omega,L^2([0,T];K))$ for every $T >0$. We define in the same way $L_{\mathscr{P},\text{loc}}^p(\Omega, \mathscr{C}([0,\infty);K))$.

In the following, we consider a complete probability space $(\Omega, \F, \Pb)$ and a cylindrical Wiener process denoted by $(W_t)_{t \geq 0}$  with values in $\Xi$, which is a real and separable Hilbert space. $(\F_t)_{t \geq 0 }$ will denote the natural filtration of $W$ augmented with the family of $\Pb$-null sets of $\F$. $H$ denotes a real and separable Hilbert space in which the SDE will take values.

\section{Preliminaries}
We will need some results about the solution of the SDE when a perturbation term $F$ is in the drift.
\subsection{The perturbed forward SDE}

Let us consider the following mild stochastic differential equation for an unknown process $(X_t)_{t \geq 0} $ with values in $H$:
\begin{align}\label{SDE}
X_t = e^{tA}x + \int_0^t e^{(t-s)A} F(s,X_s) \der s + \int_0^t e^{(t-s)A} G\der W_s, ~~\forall t \geq 0,~~ \Pb-a.s.
\end{align}
Let us introduce the following hypothesis.
\begin{hypo}\label{hypo SDE non degenerate}
\begin{enumerate}
\item $A$ is an unbounded operator $A : D(A) \subset H \rightarrow H$, with $D(A)$ dense in $H$. We assume that $A$ is dissipative and generates a stable $C_0$-semigroup $\left\{ e^{tA}  \right\}_{t \geq 0}$. By this we mean that there exist constants $\eta > 0$ and $M > 0$ such that 
\begin{align*}
\langle Ax,x \rangle \leq -\eta |x|^2, ~~ \forall x \in D(A); ~~~ |e^{tA}|_{L(H,H)} \leq Me^{-\eta t}, ~~\forall t \geq 0.
\end{align*}
\item For all $s > 0$, $e^{sA}$ is a Hilbert-Schmidt operator. Moreover $|e^{sA}|_{L_2(H,H)} \leq M s^{-\gamma}$  with $\gamma \in [0,1/2)$.
\item $F : \R_+ \times H \rightarrow H$ is bounded and measurable.
\item $G$ is a bounded linear operator in $L(\Xi,H)$.
\item $G$ is  invertible. We denote by $G^{-1}$ its bounded inverse given by Banach's Theorem. 
\end{enumerate}
\end{hypo}
\begin{rema}\label{remarque sur e^{sA} G(x) Hilbert Schmidt}
Note that under the previous set of hypotheses, we immediately get that:
\begin{align*}
|e^{sA}G|^2_{L_2(\Xi,H)} &\leq |G|_{L(\Xi,H)}^2 |e^{sA}|_{L_2(H,H)}^2 \\
& \leq |G|_{L(\Xi,H)}^2|e^{\frac{s}{2}A}|^2_{L(H,H)}|e^{\frac{s}{2}A}|_{L_2(H,H)}^2 \\
& \leq M^2 e^{-\eta s} \left(\frac{s}{2}\right)^{-2 \gamma},
 \end{align*}
 which shows that for every $s >0$ and $x\in H$,  $e^{sA}G \in L_2(\Xi,H)$, which can be used to control the stochastic integral over the time.
\end{rema}

\begin{definition}
We say that the SDE (\ref{SDE}) admits a martingale solution if there exists a new $\F$-Wiener process $(\widehat{W}^x)_{t \geq 0}$ with respect to a new probability measure $\widehat{\Pb}$ (absolutely continuous with respect to $\Pb$), and  an $\F$-adapted process $\widehat{X}^x$ with continuous trajectories for which $(\ref{SDE})$ holds with $W$ replaced by $\widehat{W}$.
\end{definition}

\begin{lemma}\label{lemma existence uniqueness SDE}
Assume that Hypothesis \ref{hypo SDE non degenerate} (1.)-(4.) holds and that $F$ is bounded and Lipschitz in $x$. Then for every $p \in [2, \infty)$, for every $T > 0$ there exists a unique process $X^x \in \LpC$ solution of (\ref{SDE}). Moreover,  
\begin{align}\label{estimate E sup |X_t|}
&\sup_{0 \leq t < + \infty }\E  |X_t^x|^p \leq C(1+|x|)^p,\\
&\label{estimate E sup |X_t| 2} \mathbb{E}\left[\sup_{0 \leq t \leq T} |X_t^x|^p\right] \leq C(1+T)(1+|x|^p),
\end{align}
for some constant $C$ depending only on $p, \gamma, M$ and $\sup_{t\geq 0}\sup_{x\in H}|F(t,x)|$. 

If $F$ is only bounded and measurable, then the solution to equation ($\ref{SDE}$) still exists but in the martingale sense. Moreover (\ref{estimate E sup |X_t|}) and (\ref{estimate E sup |X_t| 2}) still hold (with respect to the new probability). Finally such a martingale solution is unique in law.
\end{lemma}
\begin{proof}
For the first part of the lemma see \cite{DA_PRATO_ZABCZYK_STOCHASTIC_EQUATIONS_INFINITE_DIMENSION}, Theorem 7.4. For the estimate (\ref{estimate E sup |X_t|}) see Appendix $A.1$ in \cite{DEBUSSCHE_HU_TESSITORE_EBSDE_WEAK_DISSIPATIVE}. The end of the lemma is a simple consequence of the Girsanov Theorem. We will now show the estimate \eqref{estimate E sup |X_t| 2}. The ideas of this proof are adapted from \cite{FUHRMAN_TESSITORE_NONLINEAR_KOLMOGOROV_INFINITE_DIMENSION} but under our assumptions, we obtain an interesting bound depending polynomially on $T$.  We have
\begin{eqnarray*}
\sup_{0 \leq t \leq T} |X_t^x|^p &\leq& C\left(|x|^p+ C\sup_{0 \leq t \leq T}\left(\int_0^t e^{-\eta (t-s)} ds\right)^p+\sup_{0 \leq t \leq T} \left|\int_0^t e^{(t-s)A}GdW_s\right|^p \right)\\
&\leq & C\left(1+|x|^p+\sup_{0 \leq t \leq T} \left|\int_0^t e^{(t-s)A}GdW_s\right|^p \right).
\end{eqnarray*}
Let us introduce 
$$c_{\alpha}^{-1} = \int_s^t (t-u)^{\alpha-1}(u-s)^{-\alpha}du$$
with $\alpha \in ]1/p,1/2 -\gamma[$: we can assume that $p$ is large enough and then for small $p$ we will just use Jensen inequality to obtain the result. Then, the classical factorization method gives us
\begin{eqnarray*}
 \int_0^t e^{(t-s)A}GdW_s &=& c_{\alpha}\int_0^t \int_s^t (t-u)^{\alpha-1}(u-s)^{-\alpha} du e^{(t-s)A}G dW_s\\
 &=& c_{\alpha} \int_0^t (t-u)^{\alpha-1}\int_0^u (u-s)^{-\alpha} e^{(t-s)A}G dW_s du\\
 &=& c_{\alpha} \int_0^t (t-u)^{\alpha-1}e^{(t-u)A}Y_udu,
\end{eqnarray*}
with 
$$Y_u = \int_0^u (u-s)^{-\alpha} e^{(u-s)A} GdW_s.$$
We apply H\"older's inequality to obtain, with $q$ the conjugate exponent of $p$ $\left(\textrm{i.e. } \frac{1}{p} + \frac{1}{q} = 1 \right)$,
\begin{eqnarray*}
 \left|\int_0^t e^{(t-s)A}GdW_s\right|^p&\leq& C \left( \int_0^t (t-u)^{(\alpha-1)q}e^{-(t-u)\eta q}du\right)^{p/q}\left( \int_0^t |Y_u|^p du \right)\\
 &\leq & C \left( \int_0^t s^{(\alpha-1)q}e^{-\eta q s}ds\right)^{p/q}\left( \int_0^T |Y_u|^p du \right)\\
 &\leq& C  \int_0^T |Y_u|^p du. 
\end{eqnarray*}
Thus we obtain, thanks to the BDG inequality,
\begin{eqnarray*}
 \mathbb{E} \left[\sup_{0 \leq t \leq T} \left|\int_0^t e^{(t-s)A}GdW_s\right|^p\right] & \leq & C \int_0^T \mathbb{E} \left[  |Y_u|^p  \right] du\\
 &\leq & C T \sup_{0 \leq u \leq T} \mathbb{E} \left[  |Y_u|^p  \right] \\
 &\leq & C T \sup_{0 \leq u \leq T} \left(\int_0^u (u-s)^{-2\alpha-2\gamma} e^{-(u-s)\eta}ds\right)^{p/2}\\
  &\leq & C T \sup_{0 \leq u \leq T} \left(\int_0^u v^{-2\alpha-2\gamma} e^{-v\eta}dv \right)^{p/2}\\
 &\leq & C T.
\end{eqnarray*} 
\end{proof}

%We will need the following Lemma
%\begin{proposition}
% We assume Hypothesis \ref{hypo SDE non degenerate} (only points (1.)-(4.)) . Then, we have for all $p \geq 2$
%\begin{align*}
%\mathbb{E}\left[\sup_{0 \leq t \leq T} |X_t^x|^p\right] \leq C_p(1+T)(1+|x|^p).
%\end{align*}
%We stress the fact that $C_p$ is independent on $T$.
%\end{proposition}
%\begin{proof}
%\end{proof}

%\begin{corollary}\label{corrolaire E sup / T tends vers 0}
% For all $n \in \mathbb{N}^*$,
% $$\mathbb{E}\left[\sup_{0 \leq t \leq T} |X_t^x|^p\right] \leq C_{n,p}(1+T^{1/n})(1+|x|^p),$$
% and 
% $$\lim_{T \rightarrow + \infty} \frac{\mathbb{E} \left[\sup_{0 \leq t \leq T} |X_t^x|^p\right]}{T} =0.$$
%\end{corollary}
%\begin{proof}
%It is a simple consequence of previous proposition and Jensen's inequality.
%\end{proof}

We define the Kolmogorov semigroup associated to Eq. $(\ref{SDE})$ as follows: $\forall \phi : H \rightarrow \R$ measurable with polynomial growth
\begin{align*}
\Pt[\phi](x) = \E \phi(X_t^x).
\end{align*}

\begin{lemma}[Basic Coupling Estimates]\label{basic coupling estimates}
Assume that Hypothesis \ref{hypo SDE non degenerate} holds true and that $F$ is a bounded and Lipschitz function. Then there exist $\hat{c}>0$ and $\hat{\eta} > 0$ such that for all $\phi : H \rightarrow \R$ measurable with polynomial growth (i.e. $\exists C, \mu > 0$ such that $\forall x \in H$, $|\phi(x)| \leq C(1+|x|^{\mu})$), $\forall x, y \in H$,
\begin{align}\label{Basic Coupling Estimates}
|\Pt[\phi](x) - \Pt[\phi](y)| \leq \hat{c}(1+|x|^{1+\mu}+|y|^{1+\mu})e^{-\hat{\eta}t}.
\end{align}
We stress the fact that $\hat{c}$ and $\hat{\eta}$ depend on $F$ only through $\sup_{t \geq 0}\sup_{x \in H}|F(t,x)|$.
\end{lemma}
\begin{proof}
In the same manner as in the proof of Theorem $2.4$ in \cite{DEBUSSCHE_HU_TESSITORE_EBSDE_WEAK_DISSIPATIVE}, we obtain, for every $x,y \in H$,
\begin{align*}
\Pb(X_t^x \neq X_t^y ) \leq \hat{c}(1+|x|^2+|y|^2)e^{-\tilde{\eta}t}.
\end{align*}
Hence we obtain, for every $x,y \in H$ and $\phi : H \rightarrow \R$ measurable and such that $\forall x \in H$, $|\phi(x)| \leq C(1+|x|^{\mu})$, 
\begin{align*}
|\Pt[\phi](x) - \Pt[\phi](y)| &\leq \sqrt{\E(|\phi(X_t^x) - \phi(X_t^y)|^2)} \sqrt{\Pb(X_t^x \neq X_t^y)} \\
&\leq C(1+|x|^\mu + |y|^{\mu})(1+|x| + |y|)e^{-(\tilde{\eta}/2) t}\\
&\leq C(1+|x|^{1+\mu} + |y|^{1+\mu})e^{-\hat{\eta}t}.
\end{align*}
\end{proof}

\begin{corollary}\label{basic coupling estimates extended}
Relation (\ref{Basic Coupling Estimates}) can be extended to the case in which $F$ is only bounded measurable and for all $t \geq 0$,  there exists a uniformly bounded sequence of Lipschitz functions in $x$ $(F_n(t,\cdot))_{n \geq 1}$ (i.e. $\forall t \geq 0, \forall n \in \N$, $F_n(t,\cdot)$ is Lipschitz and $\sup_n \sup_t \sup_x |F_n(t,x)| < +\infty$ ) such that 
\begin{align*}
\lim_n F_n(t,x) = F(t,x),~~~~~ \forall t \geq 0, \forall x \in H.
\end{align*}
Clearly in this case in the definition of $\Pt[\phi]$ the mean value is taken with respect to the new probability $\widehat{\Pb}$.
\end{corollary}
\begin{proof}
It is enough to show that if $\mathscr{P}^n$ is the Kolmogorov semigroup corresponding to equation ($\ref{SDE}$) but with $F$ replaced by $F_n$, then  $\forall x \in H$ and $\forall t \geq 0$, 
\begin{align*}
\mathscr{P}_t^n[\phi](x) \underset{n \rightarrow + \infty}{\longrightarrow} \mathscr{P}_t[\phi](x).
\end{align*}
See the proof of Corollary $2.5$ in \cite{DEBUSSCHE_HU_TESSITORE_EBSDE_WEAK_DISSIPATIVE} for more details.
\end{proof}
\begin{rema}\label{remarque F m,n}
Similarly, if for every $t \geq 0$, there exists a uniformly bounded sequence of Lipschitz functions $(F_{m,n}(t,\cdot)_{m \in \N, n \in \N}$ (i.e. $\forall t \geq 0$, $\forall n \in \N$, $\forall m \in \N$, $F_{m,n}(t,\cdot)$ is Lipschitz and $\sup_m \sup_n \sup_t \sup_x |F_{m,n}(x)| < + \infty$) such that
\begin{align*}
\lim_n \lim_m F_{m,n}(t,x) = F(t,x),~~~~\forall t \geq 0, \forall x \in H, 
\end{align*}
then, if $\mathscr{P}^{m,n}$ is the Kolmogorov semigroup corresponding to equation ($\ref{SDE}$) but with $F$ replaced by $F_{m,n}$, we have  $\forall x \in H$ and $\forall t \geq 0$, 
\begin{align*}
\lim_n \lim_m  \mathscr{P}_t^{m,n}[\phi](x) = \mathscr{P}_t[\phi](x).
\end{align*}
\end{rema}

We will need to apply the lemma above to some functions with particular form.

\begin{lemma}\label{functions approchée par suite uniforme de fonctions Lipchitz}
Let $f : H \times \Xi^* \rightarrow \R$ be continuous in the first variable and Lipschitz in the second one and $\zeta, \zeta' : \R_+ \times H \rightarrow \Xi^*$ be such that for all $s \geq 0$, $\zeta(s,\cdot)$ and $\zeta'(s,\cdot)$ are weakly* continuous. We define 
\begin{align*}
\Upsilon(s,x) = 
\left\{
\begin{array}{ll}
\frac{f(x,\zeta(s,x)) - f(x,\zeta'(s,x))}{|\zeta(s,x)-\zeta'(s,x)|^2}(\zeta(s,x) - \zeta'(s,x))^*, & \text{if } \zeta(s,x) \neq \zeta'(s,x),\\
0, &\text{if } \zeta(s,x) = \zeta'(s,x).
\end{array}
\right.
\end{align*}
There exists a uniformly bounded sequence of Lipschitz functions $(\Upsilon_{m,n}(s,\cdot))_{m \in \N^*, n \in \N^*}$ (i.e. $\forall m \in \N^*, \forall n \in \N^*$, $\Upsilon_{m,n}(s,\cdot)$ is Lipschitz and ${\sup_m\sup_n \sup_s} \sup_x |\Upsilon_{m,n}(s,x)| < \infty$) such that 
\begin{align*}
\lim_n \lim_m \Upsilon_{m,n}(s,x) = \Upsilon(s,x), ~~~\forall s \geq 0,\forall x \in H.
\end{align*}
\end{lemma}
\begin{proof}
See the proof of Lemma $3.5$ in \cite{DEBUSSCHE_HU_TESSITORE_EBSDE_WEAK_DISSIPATIVE}.
\end{proof}

\subsection{The BSDE}
Let us fix $T >0$ and let us consider the following BSDE in finite horizon for an unknown process $(Y_s^{T,t,x},Z_s^{T,t,x})_{s \in [t,T]}$ with values in $\R \times \Xi^{*}$:
\begin{align}\label{BSDE}
Y_s^{T,t,x} = \xi^T + \int_s^T f(X_r^{t,x},Z_r^{T,t,x}) \der r - \int_s^T Z_r^{T,t,x} \der W_r, ~~~\forall s \in [t,T],
\end{align}
where $(X_s^{t,x})_{s \geq 0}$ is the mild solution of (\ref{SDE}) starting from $x$ at time $t \geq 0$. If $t=0$, we use the following standard notations $Y_s^{T,x} := Y_s^{T,0,x}$ and $Z_s^{T,x} := Z_s^{T,0,x}$.

We will assume the following assumptions.
\begin{hypo}[Path dependent case]\label{hypo BSDE 1} There exist $l > 0$, $\mu \geq 0$ such that the function $f : H \times \Xi^* \rightarrow \R$ and $\xi^T$ satisfy:
\begin{enumerate}
\item $F : H \rightarrow H$ is a Lipschitz bounded function and belongs to the class $\Ga$,
\item $\xi^T$ is an $H$ valued random variable $\F_T$ measurable and there exists $\mu \geq 0$ such that $|\xi^T| \leq C(1+\sup_{t \leq s \leq T}|X_s^x|^\mu)$,
%\item $g(\cdot)$ is continuous and have polynomial growth : for all $x \in H$, $|g(x)| \leq C(1+|x|^\mu)$,
\item $\forall x \in H$, $\forall z, z' \in \Xi^*$, $|f(x,z) - f(x,z')| \leq l |z-z'|$,
\item $f(\cdot,z)$ is continuous and $\forall x \in H$, $|f(x,0)| \leq C(1+|x|^\mu)$.
\end{enumerate}
\end{hypo}

\begin{lemma}\label{lemma existence uniqueness BSDE}
Assume that Hypotheses \ref{hypo SDE non degenerate} and \ref{hypo BSDE 1} hold true then there exists a unique solution $(Y^{T,t,x},Z^{T,t,x}) \in L_\mathscr{P}^p(\Omega, \mathscr{C}([t,T];\R)) \times {L_\mathscr{P}^p(\Omega, {L}^2([t,T];\Xi^*))}$ for all $p \geq 2$ to the BSDE (\ref{BSDE}).
% such that if we define $u_T(t,x) = Y_t^{T,t,x}$ then $u_T \in \mathscr{G}^{0,1}([0,T)\times H)$ and
%\begin{align*}
%Y_s^{T,t,x} = u_T(t,x), ~~~~Z_s^{T,t,x} = \nabla u_T(s,X_s^x)
%\end{align*}
\end{lemma}
\begin{proof}
See \cite{FUHRMAN_TESSITORE_NONLINEAR_KOLMOGOROV_INFINITE_DIMENSION}, Proposition $4.3$.
\end{proof}

We recall here the link between solutions of such BSDEs and PDEs which will justify our probabilistic approach. For this purpose we will consider the following set of Markovian hypotheses. Note that this set of hypotheses is a particular case of Hypothesis \ref{hypo BSDE 1}.
\begin{hypo}[Markovian case]\label{hypo BSDE Markov 1} There exist $l > 0$, $\mu \geq 0$ such that the function $f : H \times \Xi^* \rightarrow \R$ and $\xi^T$ satisfy:
\begin{enumerate}
\item $F : H \rightarrow H$ is a Lipschitz bounded function that belongs to the class $\Ga$,
\item $\xi^T = g(X_T^{t,x})$, where $g : H \rightarrow \R$ is continuous and have polynomial growth: for all $x \in H$, $|g(x)| \leq C(1+|x|^\mu)$,
\item $\forall x \in H$, $\forall z, z' \in \Xi^*$, $|f(x,z) - f(x,z')| \leq l |z-z'|$,
\item $f(\cdot,z)$ is continuous and $\forall x \in H$, $|f(x,0)| \leq C(1+|x|^\mu)$.
\end{enumerate}
\end{hypo}

We recall the concept of mild solution. We consider the HJB equation
\begin{align}\label{HJB cauchy}
\left\{ 
\begin{array}{ll}
\frac{\partial u(t,x)}{\partial t} + \LL u(t,x) + f(x,\nabla u(t,x)G) =0 , & \forall (t,x) \in \R_+ \times H,\\
u(T,x) = g(x),&\forall x \in H,
\end{array}
\right.
\end{align}
where $\LL u(t,x) = \frac{1}{2}\tr (G G^* \nabla^2 u(t,x)) + \langle Ax + F(x),\nabla u(t,x) \rangle$. We can define the semigroup $(\Pt)_{t \geq 0}$ corresponding to $X$ by the formula $\Pt[\phi](x) = \E\phi(X_t^x)$ for all measurable functions $\phi : H \rightarrow \R$ having polynomial growth, and we notice that $\LL$ is the formal generator of $\Pt$.
We give the definition of a mild solution of equation (\ref{HJB cauchy}):
\begin{definition}
We say that a continuous function $u : [0,T]\times H \rightarrow \R$ is a mild solution of the HJB equation (\ref{HJB cauchy}) if the following conditions hold:
\begin{enumerate}
\item $u\in \mathscr{G}^{0,1}([0,T]\times H , \R) $;
\item There exist some constant $C>0$ and some real function $k$ satisfying $\int_0^T k(t) \der t < + \infty$ such that for all $x \in H$, $h \in H$, $t \in [0,T)$ we have
\begin{align*}
|u(t,x)| \leq C(1+|x|^C), ~~~~ |\nabla u(t,x) h| \leq C|h|k(t)(1+|x|^C);
\end{align*}
\item the following equality holds:
\begin{align*}
u(t,x) = \mathscr{P}_{T-t}[g](x) + \int_t^T \mathscr{P}_{s-t}[f(\cdot,\nabla u(t,\cdot) G)](x) \der s, ~~~~\forall t \in [0,T], ~~~~\forall x \in H.
\end{align*}
\end{enumerate}
\end{definition}

\begin{lemma}
Assume that Hypotheses  \ref{hypo SDE non degenerate} and \ref{hypo BSDE Markov 1} hold true, then there exists a unique mild solution $u$ of the HJB equation (\ref{HJB cauchy}) given by the formula 
\begin{align*}
u_T(t,x) = Y_t^{T,t,x}.
\end{align*}
\end{lemma}
\begin{proof}
See Theorem 4.2 in \cite{FUHRMAN_TESSITORE_THE_BISMUT_ELWORTHY_FORMULA}.
\end{proof}

\begin{rema}
By the following change of time: $\widetilde{u}_T(t,x) := u_T(T-t,x)$, we remark that $\widetilde{u}_T(t,x)$ is the unique mild solution of
 (\ref{HJB cauchy changement de temps}). Now, remark that $\widetilde{u}_T(T,x) =  u_T(0,x) = Y_0^{T,0,x} = Y_0^{T,x}$, therefore the large time behaviour of $Y_0^{T,x}$ is the same as that of the solution of equation (\ref{HJB cauchy changement de temps}).
\end{rema}

\subsection{The EBSDE}
Let us consider the following ergodic BSDE for an unknown process $(Y_t^x,Z_t^x,\lambda)_{t \geq 0}$ with values in $\R \times \Xi^* \times\R$:
\begin{align}\label{EBSDE}
Y_t^x = Y_T^x + \int_t^T (f(X_s^x,Z_s^x) - \lambda) \der s - \int_t^T Z_s^{x} \der W_s, ~~~\forall 0 \leq t \leq T < + \infty.
\end{align}

\begin{hypo}\label{hypo EBSDE} There exist $l > 0$, $\mu \geq 0$ such that the functions $F : H \rightarrow H$ and $f : H \times \Xi^* \rightarrow \R$ satisfy:
\begin{enumerate}
\item $F : H \rightarrow H$ is a Lipschitz bounded function and belongs to the class $\Ga$,
\item $\forall x \in H$, $\forall z, z' \in \Xi^*$, $|f(x,z) - f(x,z')| \leq l |z-z'|$,
\item $f(\cdot,z)$ is continuous and $\forall x \in H$, $|f(x,0)| \leq C(1+|x|^\mu)$.
\end{enumerate}
\end{hypo}

\begin{lemma}[Existence] \label{lemma existence EBSDE}
Assume that Hypotheses \ref{hypo SDE non degenerate} and \ref{hypo EBSDE} hold true then there exists a solution $(Y^x,Z^x,\lambda) \in L_{\mathscr{P},\text{loc}}^2(\Omega, \mathscr{C}([0,\infty[;\R)) \times L_{\mathscr{P},\text{loc}}^2(\Omega, {L}^2([0,\infty[;\Xi^*)) \times \R$, to the EBSDE (\ref{EBSDE}). Moreover there exists $v : H \rightarrow \R$ of class $\Ga$ such that, for all $x,x' \in H$, for all $t \geq 0$:
\begin{align*}
&Y_t^x = v(X_t^x)~ \text{ and }~~  Z_t^x = \nabla v(X_t^x) G, \\
&v(0) = 0,\\
&|v(x) - v(x')| \leq  C(1+|x|^{1+\mu} +|x'|^{1+\mu}), \\
&|\nabla v(x)| \leq C(1+|x|^{1+\mu}).
\end{align*}
\end{lemma}
\begin{proof}
This can be proved in the same way as in \cite{DEBUSSCHE_HU_TESSITORE_EBSDE_WEAK_DISSIPATIVE}, the only difference coming from the polynomial growth of $f(x,0)$. 
\end{proof}

\begin{rema}
We stress the fact that the method used for the construction of a solution to the EBSDE requires the generator $f$ to have the following invariance property: $\forall (x, y, z) \in H\times \R \times \Xi^*, \forall c \in \R,   f(x,y+c,z) = f(x,y,z)$ as well as to have  $\forall x,y_1,y_2,z  \in H\times \R^2 \times \Xi^*, \langle f(x,y_1,z)-f(x,y_2,z) , y_1-y_2 \rangle \leq 0$. The first condition is equivalent to the fact that $f$ does not depend on $y$ which implies the second one.
\end{rema}

\begin{lemma}[Uniqueness]
The solution $(Y^x,Z^x,\lambda)$ of previous lemma is unique in the class of solutions $(Y,Z,\lambda)$ such that $Y = v(X^x)$, $|v(x)| \leq C(1+|x|^p)$ for some $p \geq 0$,  $v(0) = 0$, $Z \in L_{\mathscr{P},\text{loc}}^2(\Omega, {L}^2([0,\infty);\Xi^*))$, and $Z = \zeta(X^x)$ where $\zeta : H \rightarrow \Xi^*$ is continuous for the weak* topology.
\end{lemma}
\begin{proof}
We give a simpler proof than that in \cite{DEBUSSCHE_HU_TESSITORE_EBSDE_WEAK_DISSIPATIVE}. Indeed, let us consider two solutions $(Y^1 = v^1(X^x), Z^1$ $= \zeta^1(X^x),\lambda^1)$ and $(Y^2 = v^2(X^x), Z^2= \zeta^2(X^x) ,\lambda^2)$. From Theorem 3.10 in \cite{DEBUSSCHE_HU_TESSITORE_EBSDE_WEAK_DISSIPATIVE}, we get $\lambda^1 = \lambda^2$. Then, we have
\begin{align*}
v^{1}(x) - v^2(x) &= v^1(X_T^x) - v^2(X_T^x) + \int_0^T (f(X_s^x,Z_s^1) - f(X_s^x,Z_s^2)) \der s - \int_t^T (Z_s^1-Z_s^2) \der W_s\\
&= v^1(X_T^x) - v^2(X_T^x)  + \int_0^T (Z_s^1-Z_s^2)\frac{(f(X_s^x,Z_s^1) - f(X_s^x,Z_s^2))(Z_s^1-Z_s^2)^*}{|Z_s^1-Z_s^2|^2} \der s \\
&~~~~- \int_0^T (Z_s^1-Z_s^2) \der W_s.
\end{align*}
Now we define
\begin{align*}
\beta(x) = \left\{
\begin{array}{ll}
\frac{(f(x,\zeta^1(x)) - f(x,\zeta^2(x)))(\zeta^1(x)-\zeta^2(x))^*}{|\zeta^1(x)-\zeta^2(x)|^2},&~~ \text{ if } \zeta^1(x) - \zeta^2(x) \neq 0 \\
 0, &~~\text{otherwise}.
\end{array}\right.
\end{align*}
As $\beta(X_s^x)$ is measurable and bounded, one can apply Girsanov's theorem to deduce the existence of a new probability $\Q^T$ under which $\widetilde{W}_t = W_t - \int_0^t \beta_s \der s$, $0 \leq s \leq T$ is a Wiener process. Then
\begin{align*}
v^1(x) - v^2(x) &= \E^{{\Q}^T} [v^1(X_T^x) - v^2(X_T^x)]\\
&= \mathscr{P}_T[v^1-v^2](x)
\end{align*}
where $\mathscr{P}_t$ is the Kolmogorov semigroup associated to the following SDE
\begin{align*}
\left\{ 
\begin{array}{l}
\der U_t^x = AU_t^x \der t + F(U_t^x) \der t + G\beta(U_t^x) \der t + G \der W_t , ~~~~t \geq 0, \\
U_0^x = x.
\end{array}
\right.
\end{align*}
Now, remark that $\beta$ satisfies hypotheses of Lemma \ref{functions approchée par suite uniforme de fonctions Lipchitz}, therefore by Remark \ref{remarque F m,n},
\begin{align*}
|(v^1 - v^2)(x) - ((v^1 - v^2)(0)) |&= \left| \mathscr{P}_T[v^1-v^2](x) - \mathscr{P}_T[v^1-v^2](0) \right|\\
&\leq C(1+|x|^{p+1})e^{- \hat{\eta} T}.
\end{align*}
Then, letting $T \rightarrow + \infty$ and noting that $(v^1 - v^2)(0) = 0$ leads us to 
\begin{align*}
v^1(x) = v^2(x), ~~~~\forall x \in H.
\end{align*} 
An Itô's formula applied to $|Y^1_t-Y^2_t|^2$ is enough to show that for all $T > 0$
\begin{align*}
\E \int_0^T |Z_s^1-Z_s^2|^2 \der s = 0,
\end{align*}
which concludes the proof of uniqueness.
\end{proof}

%We recall the following uniqueness results which can be found in \cite{DEBUSSCHE_HU_TESSITORE_EBSDE_WEAK_DISSIPATIVE}.  

%\begin{lemma}[Unicity of $\lambda$] 
%Assume that Hypotheses (\ref{hypo SDE}), (\ref{hypo BSDE}) and (\ref{hypotheses EBSDE}). Moreover suppose that, for some $x \in H$, the triple 
%\end{lemma}

%For Markovian solutions, the following uniqueness result hold :
%\begin{lemma}[Unicity]
%Let $(v,\zeta)$, $(\tilde{v},\tilde{\zeta})$ be two couples of functions with $v$, $\tilde{v}$ : $H \rightarrow \R$, continuous, with $|v(x)| \leq C(1+|x|^2)$, $|\tilde{v}(x)| \leq C(1+|x|^2)$, $v(0) = \tilde{v}(0) = 0$ and $\zeta$, $\tilde{\zeta}$ continuous from $H$ to $\Xi^*$ for the weak* topology verifying $|\zeta(x)|\leq C(1+|x|^2)$, $|\tilde{\zeta(x)}|\leq C(1+|x|^2)$. 
%
%Assume that for some constants $\lambda$, $\tilde{\lambda}$ and all $x \in H$, $(v(X_t^x),\zeta(X_t^x), \lambda)$, $(\tilde{v}(X_t^x),\tilde{\zeta}(X_t^x), \tilde{\lambda})$ verify the EBSDE \ref{EBSDE}, then $\lambda = \tilde{\lambda}$, $v=\tilde{v}$, $\zeta = \tilde{\zeta}$.
%\end{lemma}

Similarly to the case of BSDE, we recall the link between solutions of such EBSDEs and ergodic HJB equations. We consider the following ergodic HJB equation for an unknown pair $(v(\cdot),\lambda)$,
\begin{align}\label{HJB ergodique}
\LL v(x) + f(x,\nabla v(x) G) - \lambda = 0, ~~~~\forall x \in H.
\end{align}
Since we are dealing with an elliptic equation it is natural to consider $(v,\lambda)$ as mild solution of equation (\ref{HJB ergodique}) if and only if, for arbitrary time $T > 0$, $v(x)$ coincides with the mild solution $u(t,x)$ of the corresponding parabolic equation having $v$ as a terminal condition:
\begin{align*}
\left\{
\begin{array}{l}
\frac{\partial u(t,x)}{\partial t } + \LL u(t,x) + f(x,\nabla u(t,x) G) - \lambda = 0,~~~~\forall t \in [0,T], ~~~~\forall x \in H,\\
u(T,x) = v(x), ~~~~\forall x \in H.
\end{array}
\right.
\end{align*}
Thus we are led to the following definition:
\begin{definition}
A pair $(v,\lambda)$, ($v : H \rightarrow \R$ and $\lambda \in \R$) is a mild solution of the HJB equation (\ref{HJB ergodique}) if the following are satisfied:
\begin{enumerate}
\item $v \in \mathscr{G}^1(H,\R)$;
\item there exists $C > 0$ such that $|\nabla v(x)| \leq C(1+|x|^C) $ for every $x \in H$;
\item for all $0 \leq t \leq T$ and $x \in H$,
\begin{align*}
v(x) = \mathscr{P}_{T-t}[v](x) + \int_t^T (\mathscr{P}_{s-t}[f(\cdot,\nabla v(\cdot) G)](x) - \lambda) \der s.
\end{align*}
\end{enumerate}
\end{definition}

We recall the following result.
\begin{lemma}
Assume that Hypotheses  \ref{hypo SDE non degenerate} and \ref{hypo EBSDE} hold true. Then equation (\ref{HJB ergodique}) admits a unique mild solution which is the pair $(v,\lambda)$ defined in Lemma \ref{lemma existence EBSDE}.
\end{lemma}
\begin{proof}
See Theorem $4.1$ in \cite{DEBUSSCHE_HU_TESSITORE_EBSDE_WEAK_DISSIPATIVE}.
\end{proof}

\section{Large time behaviour}
We recall that $(Y_s^{T,x},Z_s^{T,x})_{s \geq 0}$ denotes the solution of the finite horizon BSDE (\ref{BSDE}) with $t = 0$ and that $(Y^x_s,Z_s^x,\lambda)$ denotes the solution of the EBSDE (\ref{EBSDE}).
\subsection{First behaviour: path dependent framework and Markovian framework}
\begin{theorem}\label{theorem first estimate on behavior}
Assume that Hypotheses \ref{hypo SDE non degenerate} and \ref{hypo BSDE 1} hold true (path dependent case). Then, $\forall T > 0$, $\forall n \in \mathbb{N}^*$:
\begin{align}\label{first estimate behavior path dependent}
\left| \frac{Y_0^{T,x}}{T} - \lambda \right| \leq \frac{C_n(1+T^{1/n})(1+ |x|^{1+\mu})}{T}.
\end{align}
In particular, 
\begin{align*}
\frac{Y_0^{T,x}}{T} \underset{T \rightarrow + \infty}{ \longrightarrow} \lambda,
\end{align*}
uniformly in any bounded set of $H$.

Assume that Hypotheses \ref{hypo SDE non degenerate} and \ref{hypo BSDE Markov 1} hold true (Markovian case). Then, $\forall T > 0$:
\begin{align}\label{first estimate behavior markovien}
\left| \frac{Y_0^{T,x}}{T} - \lambda \right| \leq \frac{C(1+ |x|^{1+\mu})}{T},
\end{align}
i.e. 
\begin{align}
\left| \frac{u(T,x)}{T} - \lambda \right| \leq \frac{C(1+ |x|^{1+\mu})}{T},
\end{align}
where $u$ is the mild solution of (\ref{HJB cauchy changement de temps}).
In particular, 
\begin{align*}
\frac{u(T,x)}{T} = \frac{Y_0^{T,x}}{T} \underset{T \rightarrow + \infty}{ \longrightarrow} \lambda,
\end{align*}
uniformly in any bounded set of $H$.
\end{theorem}
\begin{proof}
First we treat the path dependent case, that is, when Hypotheses \ref{hypo SDE non degenerate} and \ref{hypo BSDE 1} hold true.
For all $x \in H$, $T > 0$:
\begin{align*}
\left| \frac{Y_0^{T,x}}{T} - \lambda \right| &\leq \left| \frac{Y_0^{T,x} - Y_0^x - \lambda T}{T} \right| + \left| \frac{Y_0^{x}}{T} \right|.
\end{align*}
We have:
\begin{align*}
Y_0^{T,x} - Y_0^x - \lambda T &= \xi^T - v(X_T^x) + \int_0^T (f(X_s^x,Z_s^{T,x}) - f(X_s^x,Z_s^{x})) \der s - \int_0^T (Z_s^{T,x} - Z_s^{x}) \der W_s\\
&=  \xi^T - v(X_T^x) + \int_0^T (Z_s^{T,x} - Z_s) \beta_s^T \der s - \int_0^T (Z_s^{T,x} - Z_s) \der W_s,
\end{align*}
where for all $s \in [0,T]$
\begin{align*}
\beta^T_s = \left\{
\begin{array}{ll}
 \frac{ (f(X_s^x,Z_s^{T,x}) - f(X_s^x,Z_s^{x}))(Z_s^{T,x} - Z_s^x)^*}{|Z_s^{T,x} - Z_s^x|^2},&~~ \text{ if } Z_s^{T,x} - Z_s^x \neq 0 \\
 0, &~~\text{otherwise}.
\end{array}\right.
\end{align*}

The process $\beta_s^T$ is progressively measurable and bounded by $l$ therefore we can apply Girsanov's Theorem to obtain that there exists a probability measure $\Q^T$ under which $\widetilde{W}_t^T = -\int_0^t \beta_s^T \der s + W_t$, $0 \leq t \leq T$ is a Wiener process. We recall that if we define 
\begin{align*}
M_T = \exp \left( \int_0^T \beta_s^T \der W_s - \frac{1}{2} \int_0^T |\beta_s^T|_{\Xi}^2\der s  \right),
\end{align*}
the following formula holds: $\der \Q^T = M_T \der \Pb$.

Taking the expectation with respect to $\Q^T$ we get
\begin{align}\label{representation}
Y_0^{T,x} - Y_0^x - \lambda T = \E^{\Q^T} (\xi^T - v(X_T^x)).
\end{align}
Hence we have
\begin{align*}
\left| \frac{Y_0^{T,x} - Y_0^x - \lambda T}{T} \right| &\leq C\frac{1+\E^{\Q^T}[\sup_{0 \leq t \leq T}|X_t^x|^\mu]}{T} + C\frac{1+\E^{\Q^T} \left(|X_T^x|^{1+\mu}\right)}{T}.
\end{align*}

The process $(X_t^x)_{t \geq 0}$ is the mild solution of 
\begin{align*}
\left\{ 
\begin{array}{l}
\der X_t^x = AX_t^x \der t + F(X_t^x) \der t + G\beta_t^T \der t + G \der \widetilde{W}_t^T , ~~~~t \in [0,T], \\
X_0^x = x.
\end{array}
\right.
\end{align*}
Thus, by Jensen's inequality and the estimate \eqref{estimate E sup |X_t| 2}, there exists a constant $C_n$ which does not depend on time such that 
\begin{align*}
\E^{\Q^T}[\sup_{0 \leq t \leq T}|X_t^x|^\mu] \leq (\E^{\Q^T}[\sup_{0 \leq t \leq T}|X_t^x|^{n\mu}])^{1/n} \leq C_n(1+T^{1/n})(1+|x|^\mu),
\end{align*}
and by Lemma \ref{lemma existence uniqueness SDE},
\begin{align*}
\E^{\Q^T}(|X_T^x|^{1+\mu}) \leq C(1+|x|^{1+\mu}),
\end{align*}
%\begin{align*}
%&\E^{\Q^T}\left[\sup_{0 \leq t \leq T} |X_t^x|^\mu\right] \leq C(1+|x|^\mu),  \\
%&\E^{\Q^T}\left[\sup_{0 \leq t \leq T}|X_t^x|^{1+\mu}\right] \leq C(1+|x|^{1+\mu}),
%\end{align*}
which allows us to obtain
\begin{align*}
\left| \frac{Y_0^{T,x} - Y_0^x - \lambda T}{T} \right| 
\leq \frac{C_n(1+T^{1/n})(1+|x|^{1+\mu})}{T}.
\end{align*}
Finally we note that 
\begin{align*}
\left| \frac{Y_0^x}{T} \right| \leq \frac{C(1+|x|^{1+\mu})}{T},
\end{align*}
which gives the result for the path dependent case.

Now we treat the Markovian case: by equality (\ref{representation}), we obtain
\begin{align}\label{representation 2}
Y_0^{T,x} - Y_0^x - \lambda T = \E^{\Q^T} (g(X_T^x) - v(X_T^x)).
\end{align}
Therefore, since $|g(x)-v(x)| \leq C(1+|x|^{1+\mu})$,
\begin{align*}
\left| \frac{Y_0^{T,x} - Y_0^x - \lambda T}{T} \right| &\leq  C\frac{1+\E^{\Q^T} \left(|X_T^x|^{1+\mu}\right)}{T} \leq C\frac{1+|x|^{1+\mu}}{T},
\end{align*}
which gives the result.
\end{proof}

\begin{rema}
If $G$ is possibly degenerate, Theorem $4.1$ remains true under additional assumptions that $f$ is locally Lipschitz in $x$ (i.e. $\exists \mu \geq 0$, $\forall x,x' \in H$, $\forall z \in \Xi^*$, $|f(x,z)-f(x',z)| \leq C(1+|x|^{\mu}+|x'|^{\mu})|x-x'|$) and that $A+F$ is dissipative. In this case, we have existence of solution to the EBSDE and $\lambda$ is unique from  \cite{FUHRMAN_HU_TESSITORE_EBSDE_BANACH_SPACE}.
\end{rema}

\subsection{Second behaviour and third behaviour: Markovian framework}
In this section, we introduce a new set of hypothesis without loss of generality. Note that it is the same as Hypothesis \ref{hypo BSDE Markov 1} but with $F \equiv 0$. However we write it again for reader's convenience.
\begin{hypo}[Markovian case, $F \equiv 0$]\label{hypo BSDE Markov 2} There exist $l > 0$, $\mu \geq 0$ such that the function $f : H \times \Xi^* \rightarrow \R$ and $\xi^T$ satisfy:
\begin{enumerate}
\item $F \equiv 0$,
\item $\xi^T = g(X_T^x)$, where $g : H \rightarrow \R$ is continuous and have polynomial growth: for all $x \in H$, $|g(x)| \leq C(1+|x|^\mu)$,
\item $\forall x \in H$, $\forall z, z' \in \Xi^*$, $|f(x,z) - f(x,z')| \leq l |z-z'|$,
\item $f(\cdot,z)$ is continuous and of polynomial growth, i.e. $\forall x \in H$, $|f(x,0)| \leq C(1+|x|^\mu)$.
\end{enumerate}
\end{hypo}

\begin{rema}
Note that setting $F \equiv 0$ is not restrictive. Indeed let us recall that the purpose of this paper is to study the large time behaviour of the mild solution of
\begin{align*}
\left\{ 
\begin{array}{ll}
\frac{\partial u(t,x)}{\partial t} = \LL u(t,x) + f(x,\nabla u(t,x)G) , & \forall (t,x) \in \R_+ \times H,\\
u(0,x) = g(x),&\forall x \in H.
\end{array}
\right.
\end{align*}
Now remark that
\begin{align*}
\langle Ax + F(x) , \nabla u(t,x) \rangle + f(x,\nabla u(t,x) G) = \langle Ax , \nabla u(t,x) \rangle + \widetilde{f}(x, \nabla u(t,x) G),
\end{align*}
where $\widetilde{f}(x,z) = f(x,z) + \langle F(x) , z G^{-1} \rangle$ is a continuous function in $x$ with polynomial growth in $x$ and Lipschitz in $z$.
Therefore, under our assumptions, we can always consider the case $F \equiv 0$ by replacing $f$ by $\widetilde{f}$ if necessary.
\end{rema}

\begin{theorem}\label{theorem second estimate on behavior}
Assume that Hypotheses \ref{hypo SDE non degenerate} and \ref{hypo BSDE Markov 2} hold true. Then there exists $L \in \R$ such that,
\begin{align*}
\forall x \in H, ~~~Y_0^{T,x} - \lambda T - Y_0^x  \underset{T \rightarrow +\infty}{\longrightarrow} L,
\end{align*}
i.e. 
\begin{align*}
\forall x \in H, ~~~u(T,x) - \lambda T - v(x) \underset{T \rightarrow +\infty}{\longrightarrow} L,
\end{align*}
where $u$ is the mild solution of (\ref{HJB cauchy changement de temps}).

Furthermore the following rate of convergence holds 
\begin{align*}
|Y_0^{T,x} - \lambda T - Y_0^x - L| \leq C(1 + |x|^{2+\mu} )e^{-\hat{\eta} T},
\end{align*}
i.e.
\begin{align*}
|u(T,x)- \lambda T - v(x) - L| \leq C(1 + |x|^{2+\mu} )e^{-\hat{\eta} T},
\end{align*}
where $u$ is the mild solution of (\ref{HJB cauchy changement de temps}).
\end{theorem}

\begin{proof}
Let us start by defining
\begin{align*}
&u_T(t,x) := Y_t^{T,t,x}\\
&w_T(t,x) := u_T(t,x) - \lambda (T-t) - v(x).
\end{align*}
We recall that $Y_s^{T,t,x} = u_T(s,X_s^{t,x})$ and that $Y_s^x = v(X_s^x)$, where $v$ is defined in Lemma \ref{EBSDE}. We recall that for all $T,S \geq 0$, $u_T$ is the unique mild solution of 
\begin{align*}
\left\{ 
\begin{array}{ll}
\frac{\partial u_T(t,x)}{\partial t} + \LL u_T(t,x) + f(x,\nabla u_T(t,x)G) = 0, & \forall (t,x) \in [0,T] \times H,\\
u_T(T,x) = g(x),&\forall x \in H,
\end{array}
\right.
\end{align*}
and that 
 $u_{T+S}$ is the unique mild solution of 
\begin{align*}
\left\{ 
\begin{array}{ll}
\frac{\partial u_{T+S}(t,x)}{\partial t} + \LL u_{T+S}(t,x) + f(x,\nabla u_{T+S}(t,x)G) = 0, & \forall (t,x) \in [0,T+S] \times H,\\
u_{T+S}(T+S,x) = g(x),&\forall x \in H.
\end{array}
\right.
\end{align*}
This implies that $u_T(0,x)=u_{T+S}(S,x)$, for all $x \in H$, and then,
\begin{align}\label{changement de temps w}
w_T(0,x) = w_{T+S}(S,x).
\end{align} 

We will need some estimates on $w_T$ given in the following lemma.
\begin{lemma}\label{lemma estimates on w_T(x)}
Under the hypotheses of Theorem \ref{theorem second estimate on behavior}, $\exists C>0$, $\forall x,y \in H$, $\forall T > 0$, $\forall 0 < T' \leq T$, $\exists C_{T'}>0$,
\begin{align*}
&|w_T(0,x)| \leq C(1+|x|^{1+\mu}),\\
&|\nabla_x w_T(0,x)| \leq {\frac{C_{T'}}{\sqrt{T'}}}(1+|x|^{1+\mu}),\\
&|w_T(0,x) - w_T(0,y)| \leq  C(1+|x|^{2+\mu}+|y|^{2+\mu})e^{-\hat{\eta}T}.
\end{align*}
We stress the fact that $C$ depends only on $\eta$, $M$, $\gamma$, $G$, $\mu$, $\sup_{x \in H} \frac{|g(x)|}{1+|x|^\mu}$, $\sup_{x \in H}\frac{|f(x)|}{1+|x|^\mu}$ and $l$ and $C_T'$ depends on same parameters as $C$ and $T'$.
\end{lemma}
\begin{proof}[Proof of Lemma \ref{lemma estimates on w_T(x)} ]
The first inequality of the lemma is a direct application of estimate in Theorem \ref{theorem first estimate on behavior}. Indeed, $\forall x \in H, \forall T >0$,
\begin{align}\label{estimee w_T(0,x) dans la preuve}
|w_T(0,x)| &= |u_T(0,x)-\lambda T - v(x)| \nonumber \\
& = |Y_0^{T,x} - Y_0^x- \lambda T |\nonumber \\
& \leq C(1+|x|^{1+\mu}).
\end{align}

Now, let us establish the gradient estimate. The process $(w_T(s,X_s^{t,x}))_{t \leq s \leq T}$ satisfies the following equation, for all $t \leq s \leq T$,
\begin{align*}
w_T(s,X_s^{t,x}) =  w_T(T,X_T^{t,x}) &+ \int_s^T ({f}(X_r^{t,x},Z_r^{T,t,x}) - {f}(X_r^{t,x},Z_r^{t,x})) \der r \\
& ~~~~~~~~~~~~~~~~~~~~- \int_s^T (Z_r^{T,t,x} - Z_r^{t,x}) \der W_r.
\end{align*}

Now remark that for all $t \leq T$ and $t \leq s \leq   T' \leq T$ we have
\begin{align*}
w_T(s,X_s^{t,x})  &= w_T(T',X_{T'}^{t,x}) + \int_s^{T'} ({f}(X_r^{t,x},Z_r^{T,t,x}) - {f}(X_r^{t,x},Z_r^{t,x})) \der r \\
&~~~~~~~~~~~~~~~~~~~~~- \int_s^{T'} (Z_r^{T,t,x} - Z_r^{t,x}) \der {W}_r \\
& = w_{T-T'}(0,X_{T'}^{t,x}) +  \int_s^{T'} ({f}(X_r^{t,x},Z_r^{T,t,x}-Z_r^{t,x}+Z_r^{t,x}) - {f}(X_r^{t,x},Z_r^{t,x})) \der r \\
& ~~~~~~~~~~~~~~~~~~~~~- \int_t^{T'} (Z_r^{T,t,x} - Z_r^{t,x}) \der {W}_r,
\end{align*}
where we have used equality (\ref{changement de temps w}) for the second equality.

We also recall that (see \cite{FUHRMAN_TESSITORE_THE_BISMUT_ELWORTHY_FORMULA} Theorem $4.2$ and \cite{DEBUSSCHE_HU_TESSITORE_EBSDE_WEAK_DISSIPATIVE} Theorem $3.8$), $\forall x  \in H, \forall s \in [t,T[$,
\begin{equation*}
Z_s^{T,t,x} = \nabla_x u_T(s,X_s^{t,x}) G, \quad \textrm{and} \quad Z_s^{t,x} = \nabla_x v(X_s^{t,x}) G. 
\end{equation*}
Then we easily obtain that 
\begin{align*}
Z_r^{T,t,x}-Z_r^{t,x} = \nabla_x w_T(r,X_r^{t,x}) G.
\end{align*}
Thus, applying the Bismut-Elworthy formula (see \cite{FUHRMAN_TESSITORE_THE_BISMUT_ELWORTHY_FORMULA}, Theorem $4.2$), we get 
$\forall x, h \in H$, $\forall t < T$, 
\begin{align*}
\nabla_x w_T(t,x)h &= \E \int_t^{T'} [{f} \left(X_s^{t,x},\nabla_x w_T(s,X_s^{t,x}) G + Z_s^{t,x}\right) - {f} \left(X_s^{t,x},Z_s^{t,x}\right)] U^{h}(s,t,x) \der s \\
&~~~~~~~~~~~~~~~~~~~~~~~ + \E\left[[w_{T-T'}(0,X_{T'}^{t,x})] U^{h}(T',t,x)\right],
\end{align*}
where, $\forall 0 \leq s \leq T$, $\forall x \in H$,
\begin{align*}
U^{h}(s,t,x) = \frac{1}{s-t} \int_t^s \langle G^{-1}\nabla_x X_u^{t,x} h, \der W_u \rangle.
\end{align*}
Let us recall that 
\begin{align*}
&\nabla_x X_s^{t,x} h = e^{(s-t)A} h,
\end{align*}
then,
\begin{equation*}
\E|U^{h}(s,t,x)|^2  = \frac{1}{|s-t|^2}\int_t^s |G^{-1} \nabla_x X_u^{t,x} h |^2 \der u \leq \frac{C |h|^2}{s-t},
\end{equation*}
where $C$ is independent of $t,s$ and $x$.

Thus we get,$\forall x, h \in H$, $\forall t < T$ , using inequality (\ref{estimee w_T(0,x) dans la preuve}),
\begin{align*}
|\nabla_x w_T(t,x) h| \leq C \int_t^{T'} \frac{\sqrt{\E(|\nabla_x w_T(s,X_s^{t,x})|^2)} |h|}{\sqrt{s-t}} \der s + C\frac{(1+|x|^{1+\mu}) |h|}{\sqrt{T'-t}}.
\end{align*}
We define 
\begin{align*}
\varphi(t) = \sup_{x \in H} \frac{|\nabla_x w_T(t,x)|}{(1+|x|^{1+\mu})},
\end{align*}
and we remark that $\varphi(t)$ is well defined for all $t < T$. Indeed $\nabla_x w_T(t,x) = \nabla_x u_T(t,x) - \nabla_x v(x)$ and we have $|\nabla_x u_T(t,x)| \leq C_T(T-t)^{-1/2}(1+|x|^{\mu})$ (see Theorem $4.2$ in \cite{FUHRMAN_TESSITORE_NONLINEAR_KOLMOGOROV_INFINITE_DIMENSION}) and  $|\nabla_x v(x)| \leq C(1+|x|^{1+\mu})$ (by Lemma \ref{lemma existence EBSDE}).
Then we obtain
\begin{align*}
|\nabla_x w_T(t,x) h| &\leq C\int_t^{T'} \frac{\varphi(s)}{\sqrt{s-t}} \sqrt{\E((1+|X_s^{t,x}|^{1+\mu})^2)}|h| \der s + C\frac{(1+|x|^{1+\mu}) |h|}{\sqrt{T'-t}}
\end{align*}
which leads to 
\begin{align*}
\frac{|\nabla_x w_T(t,x) h|}{(1+|x|^{1+\mu})} \leq C|h|\left(\int_t^{T'} \frac{\varphi(s)}{\sqrt{s-t}} \der s + \frac{1}{\sqrt{T'-t}}\right) .
\end{align*}
Taking the supremum over $h$ such that $|h|=1$ and $x \in H$, we have
\begin{align*}
\varphi(t) \leq C\int_t^{T'} \frac{\varphi(s)}{\sqrt{s-t}} \der s + \frac{C}{\sqrt{T'-t}}.
\end{align*}
Now remark that we can rewrite the above inequality as follows:
\begin{align*}
\varphi(T'-t) \leq C\int_0^t \frac{\varphi(T'-s)}{\sqrt{t-s}} \der s + \frac{C}{\sqrt{t}},
\end{align*}
then by Lemma $7.1.1$ in \cite{HENRY_LIVRE} we get
\begin{align*}
\varphi(T'-t) \leq \frac{C}{\sqrt{t}} + C\theta \int_0^t E'(\theta(t-s))\frac{1}{\sqrt{s}} \der s,
\end{align*}
where 
\begin{align*}
\theta = (C\Gamma(1/2))^2,\quad \textrm{and} \quad E(z) = \sum_0^\infty \frac{z^{n/2}}{\Gamma (n/2 + 1)}.
\end{align*} 
Therefore, taking $t = T'$ leads us to
\begin{align*}
\varphi(0) \leq \frac{C}{\sqrt{T'}} + C\theta \int_0^{T'} E'(\theta(T'-s))\frac{1}{\sqrt{s}} \der s \\
\end{align*}
which implies
\begin{align*}
|\nabla_x w_T(0,x)| \leq C_{T'}(1+|x|^{1+\mu})\frac{1+\sqrt{T'}}{\sqrt{T'}}.
\end{align*}

For the third inequality of Lemma \ref{lemma estimates on w_T(x)}, we have in the same way as for equation ($\ref{representation 2}$), $\forall x \in H$, $\forall T >0$,
\begin{align*}
w_T(0,x) &= \E^{\Q^T} (g(X_T^x)-v(X_T^x)) \\
& = \mathscr{P}_T[g-v](x),
\end{align*}
where $\mathscr{P}_t$ is the Kolmogorov semigroup associated to the following SDE
\begin{align*}
\der U_t^x = [AU_t^x + G\beta^T(t,U_t^x)] \der t + G \der W_t, ~~~U_0^x =x, \quad t\geq 0,
\end{align*} 
and where $\beta^T(t,x) =$
\begin{align*}
 \left\{
\begin{array}{ll}
\frac{(f(x,\nabla u_T(t,x)G) - f(x,\nabla v(x)G))(\nabla u_T(t,x)G-\nabla v(x)G)^*}{|(\nabla u_T(t,x)-\nabla v(x))G|^2} \mathds{1}_{t < T}& \\
~~~+ \frac{(f(x,\nabla u_T(T,x)G) - f(x,\nabla v(x)G))(\nabla u_T(T,x)G-\nabla v(x)G)^*}{|(\nabla u_T(T,x)-\nabla v(x))G|^2} \mathds{1}_{t \geq T}& \text{ if } \nabla u_T(t,x)-\nabla v(x) \neq 0,\\
0~, &~~\text{otherwise}.
\end{array}\right.
\end{align*}

Therefore, $\forall x \in H$, $\forall T >0$ we can write
\begin{align*}
|w_T(0,x) - w_T(0,y)| &= |\mathscr{P}_T[g-v](x)-\mathscr{P}_T[g-v](y)|.
\end{align*}
Then, as $\beta^T$ is uniformly bounded in $t$ and $x$, by Lemma \ref{functions approchée par suite uniforme de fonctions Lipchitz},
and thanks to Remark \ref{remarque F m,n} 
we obtain, since $(g(\cdot)-v(\cdot))$ has polynomial growth of order $1+\mu$:
\begin{align}\label{estimate w_T(x) - w_T(y) 1}
|w_T(0,x) - w_T(0,y)| \leq  C(1+|x|^{2+\mu}+|y|^{2+\mu})e^{-\hat{\eta}T},
\end{align}
which conclude the proof of the lemma.
\end{proof}
Now, let us come back to the proof of the theorem. The first estimate of Lemma \ref{lemma estimates on w_T(x)} allows us to construct, by a diagonal procedure, a sequence $(T_i)_i \nearrow +\infty$ such that for a function $w:D \rightarrow \mathbb{R}$ defined on a countable dense subset $D$ of $H$, the following holds
\begin{align*}
\forall x \in D, \lim_{i \rightarrow + \infty} w_{T_i}(0,x) = w(x).  
\end{align*}
Then we fix $T'>0$ and, by second estimate of  Lemma \ref{lemma estimates on w_T(x)}, we obtain that for every $x,y \in H$, for every $T \geq T'$,
\begin{align*}
|w_T(0,x)-w_T(0,y)| \leq \frac{C_{T'}}{\sqrt{T'}}(1+|x|^{1+\mu}+|y|^{1+\mu})|x-y|.
\end{align*}
By using this last inequality it is possible to extend $w$ to the whole $H$. Indeed if $x \notin D$ then there exists $(x_p)_{p \in \N} \in D^\N$ such that $x_p \rightarrow x$. Thus if we set $w(x) := \lim_{p \rightarrow + \infty} w(x_p)$, it is easy to check that $w_T(0,x) \underset{T \rightarrow + \infty}{\longrightarrow} w(x)$ for any $x \in H$.

Now, let us show that $w : H \rightarrow \R$ is a constant function. We have, by the third inequality of Lemma \ref{lemma estimates on w_T(x)}, for all  $ x,y \in H$ and $ T > 0$,
\begin{align*}
|w_T(0,x)-w_T(0,y)| \leq C(1+|x|^{{2+\mu}} + |y|^{{2+\mu}})e^{-\hat{\eta} T}.
\end{align*}

Applying the previous inequality with $T = T_i$ and taking the limit in $i$ shows us that $x \mapsto w(x)$ is a constant function, namely  there exists $L_1 \in \R$ (independent of $x$) such that: $\forall x \in H$,
\begin{align*}
\lim_i w_{T_i}(0,x) = L_1.
\end{align*}

We remark that for any compact subset $K$ of $H$, $\left\{ w_T(0,\cdot)_{|K} ;T >1\right\}$ is a relatively compact subspace of the space of continuous functions $K \rightarrow \R$ for the uniform distance (denoted by $(\mathscr{C}(K,\R),||\cdot||_{K,\infty})$ thanks to the two first inequalities of Lemma \ref{lemma estimates on w_T(x)}. Note now that $L_1$ is an accumulation point of $\left\{ w_T(0,\cdot)_{|K} ;T >1\right\}$ since $w_{T_i}(\cdot)$ converges uniformly toward $L_1$ on any compact subset of $H$ by the second inequality of Lemma \ref{lemma estimates on w_T(x)}.

Therefore, if we show that for every compact subset $K$ of $H$, $\left\{ w_T(0,\cdot)_{|K} ;T >1\right\}$ admits only one accumulation point (independently of $K$), it will imply that for all $K$ compact subset of $H$, for all $x \in K$
\begin{align*}
\lim_{T \rightarrow + \infty} w_T(0,x) = L_1,
\end{align*}
or, in other words, for all $x \in H$,
\begin{align*}
\lim_{T \rightarrow + \infty} w_T(0,x) = L_1.
\end{align*}
Now we claim that the accumulation point is unique. Let us assume that there exists another subsequence $(T_i')_{i \in \N} \nearrow + \infty$ and $w_{\infty,K}(\cdot) \in \mathscr{C}(K,\R)$such that 
\begin{align*}
\parallel w_{T'_i}(0,\cdot) - w_{\infty,K}(\cdot) \parallel_{K,\infty} \underset{i \rightarrow + \infty}{\longrightarrow} 0.
\end{align*}
Then, by the third inequality of Lemma \ref{lemma estimates on w_T(x)}, there exists $L_{2,K}$ such that $\forall x \in K$, $w_{\infty,K}(x) = L_{2,K}$.

Let us write, $\forall x \in H$, $\forall T,S >0$:
\begin{align*}
w_{T+S}(0,x) &= Y_0^{T+S,x} - \lambda(T+S) - Y_0^x \\
&= Y_S^{T+S,x} -\lambda T - Y_S^x + \int_0^S (f(X_r^x,Z_r^{T+S,x}) - f(X_r^x,Z_r^x)) \der r \\
&~~~~~~~~~~~~~~~~~~~~~~~~~~~~~ - \int_0^S (Z_r^{T+S,x} - Z_r^x) \der W_r \\
&= Y_S^{T+S,x} -\lambda T - Y_S^x + \int_0^S (Z_r^{T+S,x} - Z_r^{x}) \der \widetilde{W}_r^{T,S},
\end{align*}
with
\begin{align*}
\widetilde{W}_t^{T,S} = -\int_0^t \beta^{T,S}(s,X_s^x) \der s + W_t, ~~~~\forall t \in [0,S]
\end{align*}
and where $\beta^{T,S}(t,x) =$
\begin{align*}
\left\{
\begin{array}{ll}
\frac{(f(x,\nabla u_{T+S}(t,x)G) - f(x,\nabla v(x)G))((\nabla u_{T+S}(t,x)-\nabla v(x))G)^*}{|(\nabla u_{T+S}(t,x)-\nabla v(x))G|^2} \mathds{1}_{t \leq S},&\\
~+\frac{(f(x,\nabla u_{T+S}(S,x)G) - f(x,\nabla v(x)G))((\nabla u_{T+S}(S,x)-\nabla v(x))G)^*}{|(\nabla u_{T+S}(S,x)-\nabla v(x))G|^2} \mathds{1}_{t > S} & \text{if } \nabla u_{T+S}(t,x)-\nabla v(x) \neq 0 \\
 0, & \text{otherwise}.
\end{array}\right.
\end{align*}

Taking the expectation with respect to the probability $\Q^{T,S}$ under which ${W}^{T,S}$ is a Brownian motion we get (using equality (\ref{changement de temps w}) for the third equality):
\begin{align}\label{w_(T+S)(0,x)=E(w_T(0,U_S^x))}
w_{T+S}(0,x)  &= \E^{\Q^{T,S}}(Y_S^{T+S,x} - \lambda T - Y_S^x) \nonumber\\
&= \E^{\Q^{T,S}}(w_{T+S}(S,X_S^x))\nonumber\\
&= \E^{\Q^{T,S}}(w_T(0,X_S^x)) \nonumber\\
&= \mathscr{P}_S[w_T(0,\cdot)](x),
\end{align}
where $\mathscr{P}_t$ is the Kolmogorov semigroup of the following SDE defined $\forall t \in \R_+$:
\begin{align*}
\der U_t^x = [AU_t^x + G\beta^{T,S}(t,U_t^x) ] \der t + G \der W_t, ~~~U_0^x =x.
\end{align*}
This implies, substituting $T$ by $T_i'$ and $S$ by $T_i - T_i'$, (up to a subsequence for $(T_i)_{i\in \N}$ such that $T_i > T_i'$), for all $x \in H$,
\begin{align*}
w_{T_i}(0,x) = \mathscr{P}_{T_i-T_i'}[w_{T_i'}(0,\cdot)](x).
\end{align*}
We recall that $\lim_{i} w_{T_i}(0,x) = L_1$ and we will show that the second term converges toward $L_{2,K}$ when $x \in K$. We have, for all $x \in K$,
\begin{align*}
|\mathscr{P}_{T_i-T_i'}[w_{T_i'}(0,\cdot)](x)-L_{2,K}| &\leq |\mathscr{P}_{T_i-T_i'}[w_{T_i'}(0,\cdot)](x) -w_{T_i'}(0,x)|+ |w_{T_i'}(0,x) - L_{2,K}| \\
\end{align*}
We recall that $|w_{T_i'}(0,x) - L_{2,K}| \underset{i \rightarrow + \infty}{\longrightarrow} 0$. Furthermore, if we denote by $U^{x,m,n}$ the mild solution of 
\begin{align*}
\der U_t^{x,m,n} = [AU_t^{x,m,n}+G\beta_{m,n}^{T,S}(t,U_t^{x,m,n})]\der t + G \der W_t,~~~~~ U_0^{x,m,n}=x,
\end{align*}
where $(\beta_{m,n}^{T,S})_{m \in \N^*, n \in \N^*}$ is the sequence of functions obtained by Lemma \ref{functions approchée par suite uniforme de fonctions Lipchitz}, then we have
\begin{align*}
|\mathscr{P}_{T_i-T_i'}[w_{T_i'}(0,\cdot)](x)-w_{T_i'}(0,x)| &= |\lim_n \lim_m \E (w_{T_i'}(0,U_{T_i-T_i'}^{x,m,n}))-w_{T_i'}(0,x)|\\
&\leq \lim_n \lim_m |\E (w_{T_i'}(0,U_{T_i-T_i'}^{x,m,n}))-w_{T_i'}(0,x)|\\
&\leq \lim_n \lim_m \E |w_{T_i'}(0,U_{T_i-T_i'}^{x,m,n}) - w_{T_i'}(0,x)|\\
&\leq \lim_n \lim_m \E[C(1+|U_{T_i-T_i'}^{x,m,n}|^{2+\mu}+|x|^{2+\mu})e^{-\hat{\eta}T_i'}]\\
&\leq C(1+|x|^{2+\mu})e^{-\hat{\eta}T_i'}
\end{align*}
where the third line is obtained thanks to the third inequality of Lemma \ref{lemma estimates on w_T(x)}. Therefore, letting $i \rightarrow + \infty$ shows us that for all $x \in K$,
\begin{align*}
\mathscr{P}_{T_i-T_i'}[w_{T_i'}(0,\cdot)](x) \longrightarrow L_{2,K}.
\end{align*}
Thus, for any compact subset $K$ of $H$, $L_1 = L_{2,K}$, which as mentioned before, implies that for all $x \in H$,
\begin{align*}
\lim_{T \rightarrow + \infty} w_T(0,x) = L_1.
\end{align*}

Finally we prove that this convergence holds with an explicit rate of convergence. Let us write, $\forall x \in H, \forall T > 0$,
\begin{align*}
 |w_T(0,x) - L| &= \lim_{V \rightarrow +\infty} |w_T(0,x)-w_V(0,x)|\\
&= \lim_{V \rightarrow +\infty} |w_T(0,x)-\mathscr{P}_{V-T}[w_T(0,\cdot)](x)|
\end{align*}
thanks to equality (\ref{w_(T+S)(0,x)=E(w_T(0,U_S^x))}), where $\mathscr{P}_t$ is Kolmogorov semigroup associated to the following SDE defined $\forall t \in \R_+$:
\begin{align*}
\der U_t^x = [AU_t^x + \beta^{T,V-T}(t,U_t^x)] \der t + G \der W_t, ~~~U_0 =x.
\end{align*}
Now, if we denote by $U^{x,m,n}$ the mild solution of the following SDE, $\forall t \geq 0$,
\begin{align*}
\der U_t^{x,m,n} = [AU_t^{x,m,n}+G\beta_{m,n}^{T,V-T}(t,U_t^x)]\der t + G\der W_t, ~~~~~U_0^{x,m,n}=x,
\end{align*}
where $(\beta_{m,n}^{T,V-T})_{m \in \N^*, n \in \N^*}$ is the sequence of functions obtained by lemma \ref{functions approchée par suite uniforme de fonctions Lipchitz}, then we have,
\begin{align*}
|w_T(0,x) - L| &= \lim_{V \rightarrow + \infty}|w_T(0,x)-\lim_{n \rightarrow + \infty}\lim_{m \rightarrow + \infty} \E(w_T(0,U_{V-T}^{x,m,n}))|\\
&\leq  \lim_{V \rightarrow + \infty} \lim_{n \rightarrow + \infty}\lim_{m \rightarrow + \infty} C\E (1+|x|^{2+\mu}+|U_{V-T}^{x,m,n}|^{2+\mu})e^{-\hat{\eta}T}\\
&\leq C(1+|x|^{2+\mu})e^{-\hat{\eta}T}
\end{align*}
thanks to the third estimate in Lemma \ref{lemma estimates on w_T(x)}.
\end{proof}

\begin{rema}
By the third inequality of Lemma \ref{lemma estimates on w_T(x)}, we have
\begin{align*}
| Y_0^{T,x} - Y_0^{T,0}  - (Y_0^x - Y_0^0) | \leq C(1+|x|^{2+\mu})e^{- \hat{\eta} T},
\end{align*}
i.e.
\begin{align*}
| u(T,x)- u(T,0)  - (v(x)-v(0)) | \leq C(1+|x|^{2+\mu})e^{- \hat{\eta} T},
\end{align*}
which provides possibly an efficient approximation for $Y_0^x$ and $v(x)$.
\end{rema}

\section{Application to an ergodic control problem}
In this section, we show how we can apply our results to an ergodic control problem. In this section we will still assume that Hypothesis \ref{hypo SDE non degenerate}  hold true and that $F$ is Lipschitz and bounded and belong to the class $\Ga$. Let $U$ be a separable metric space. We define a control $a$ as an $(\F_t)$-predictable $U$-valued process. We will assume the following
\begin{hypo}\label{hypo controle}
The functions $R : U \rightarrow H$, $L : H \times U \rightarrow \R$ and $g_0 : H \rightarrow \R$ are measurable and satisfy, for some constants $c > 0$, $C>0$ and $\mu$,
\begin{enumerate}
\item $|R(a)| \leq c$, for all $a \in U$;
\item $L(\cdot,a)$ is continuous in $x$ uniformly with respect to $a \in U$; furthermore $|L(x,a)| \leq C(1+|x|^\mu)$ for all $x \in H$ and all $a \in U$;
\item $g_0(\cdot)$ is continuous and $|g_0(x)| \leq C(1+|x|^\mu)$ for all $x \in H$.
\end{enumerate}
\end{hypo}
\bigskip

We denote by $(X_t^x)_{t \geq 0}$ the solution of (\ref{SDE}).
Given an arbitrary control $a$ and $T>0$, we introduce the Girsanov density
\begin{align*}
\rho_T^{x,a} = \exp \left( \int_0^T  G^{-1}R(a_s) \der W_s - \frac{1}{2}\int_0^T |G^{-1}R(a_s)|_{\Xi}^2\der s \right)
\end{align*}
and the probability $\Pb_T^{a}=\rho_T^{a} \Pb$ on $\F_T$. We introduce two costs.
The first one is the cost in finite horizon:
\begin{align*}
J^T(x,a) := \E^{a,T} \int_0^T L(X_s^x,a_s )\der s + \E^{a,T} g_0(X_T^x),
\end{align*}
where $\E^{a,T}$ denotes the expectation with respect to $\Pb_T^{a}$.
The associated optimal control problem is to minimize the cost $J^T(x,a)$ over all controls $a^T : \Omega \times [0,T] \rightarrow U$, progressively measurable.

The second one is called the ergodic cost and is the time averaged finite horizon cost:
\begin{align*}
J(x,a) := \limsup_{T \rightarrow + \infty} \frac{1}{T} \E^{a,T} \int_0^T L(X_s^x,a_s )\der s.
\end{align*}
The associated optimal control problem is to minimize the cost $J(x,a)$ over all controls $a : \Omega \times [0,+\infty[ \rightarrow U$, progressively measurable.

We notice that $W_t^{a} = W_t - \int_0^t G^{-1}R(a_s)\der s$ is a Wiener process on $[0,T]$ under $\Pb_T^a$ and that 
\begin{align*}
\der X_t^x = (A X_t^x + F(X_t^x) + R(a_t))\der t + G \der W_t^a, ~~~~~t \in [0,T],
\end{align*}
and this justifies our formulation of the control problem.

%Let us denote by $(X_t^{x,a},W_t^{x,a})_{t \geq 0}$ the martingale solution  (unique in law) to the following SDE with respect to a new probability $\Pb^{x,a}$ 
%\begin{align*}
%\left\{
%\begin{array}{l}
%\der X_t^{x,a} = (AX_t^{x,a} + F(X_t^{x,a}))\der t + R(a_t)\der t + G \der W_t^{x,a},~~~~\forall t \geq 0, \\
%X_0^{x,a} = x. 
%\end{array}
%\right.
%\end{align*}
%
%We consider two costs. The first one is the cost in finite horizon:
%\begin{align*}
%J^T(x,a) := \E^{x,a} \int_0^T L(X_s^x,a_s )\der s + \E^{x,a} g_0(X_T^x),
%\end{align*}
%where $\E^{x,a}$ denotes the expectation with respect to $\Pb^{x,a}$.
%
%The associated optimal control problem is to minimize the cost $J^T(x,a)$ over all controls $a^T : \Omega \times [0,T] \rightarrow U$, progressively measurable.
%
%The second one is called the ergodic cost and is the time averaged finite horizon cost:
%\begin{align*}
%J(x,a) := \limsup_{T \rightarrow + \infty} \frac{1}{T} \E^{x,a} \int_0^T L(X_s^x,a_s )\der s.
%\end{align*}
%The associated optimal control problem is to minimize the cost $J(x,a)$ over all controls $a : \Omega \times [0,+\infty[ \rightarrow U$, progressively measurable.

We want to show how our results can be applied to such an optimisation problem to get an asymptotic expansion of the finite horizon cost involving the ergodic cost.

To apply our results, we first define the Hamiltonian in the usual way
\begin{align}\label{hamiltonien}
f_0(x,z) = \inf_{a \in U} \left\{ L(x,a) + zG^{-1}R(a) \right\},
\end{align}
and we remark that, if for all $x,z,$ the infimum is atttained in $\eqref{hamiltonien}$ then by the Filippov Theorem, see \cite{MCSHANE_WARFIELD_FILIPPOV_IMPLICIT}, there exists a measurable function $\gamma : H \times \Xi^* \rightarrow U$ such that 
\begin{align*}
f_0(x,z) = L(x,\gamma(x,z)) + zG^{-1}R(\gamma(x,z)).
\end{align*}

\begin{lemma}
Under the above assumptions, the Hamiltonian $f_0$ satisfies assumptions on $f$ in hypotheses \ref{hypo BSDE 1}, \ref{hypo BSDE Markov 1}, \ref{hypo EBSDE} or \ref{hypo BSDE Markov 2}.
\end{lemma}
\begin{proof}
See Lemma $5.2$ in \cite{FUHRMAN_TESSITORE_NONLINEAR_KOLMOGOROV_INFINITE_DIMENSION}.
%We first prove that $f_0$ is continuous in $x$. Clearly, $f_0$ is upper semicontinuous in $x$. To prove that $f_0$ is also lower semicontinuous, let $x_n \rightarrow x$ and take $a_n$ such that $L(x_n,a_n) + zG^{-1}R(a_n) \leq f_0(x,a_n) + 1/n$, then:
%\begin{align*}
%L(x,a_n) + zG^{-1}R(a_n) \leq f_0(x_n,z) + 1/n + [L(x,a_n) - L(x_n,a_n)]
%\end{align*}
%and, since $L$ is continuous in $x$ uniformly in $a \in U$, 
%\begin{align*}
%f_0(x,z) \leq \liminf_n f_0(x_n,z).
%\end{align*}
%Consequently the following hold (by upper semicontinuity for the third inequality)
%\begin{align*}
%f_0(x,z) \leq  \liminf_n f_0(x_n,z) \leq \limsup_n f_0(x_n,z) \leq f_0(x,z),
%\end{align*}
%which prove the continuity of $f_0$ in $x$.
%
%It is also clear that $|f_0(x,0)| \leq |L(x,a)| \leq C(1+|x|^\mu)$.
%
%Moreover, for all $z_1,z_2 \in \Xi^*$, 
%\begin{align*}
%f_0(x,z_1) &\leq L(x,a) + z_2G^{-1}R(a) +z_1G^{-1}R(a) - z_2G^{-1}R(a)\\
%&\leq L(x,a) + z_2G^{-1}R(a) + c|z_1-z_2|,
%\end{align*}
%taking the infimum over $a \in U$, we obtain
%\begin{align*}
%f_0(x,z_1)  - f_0(x,z_2) \leq c|z_1-z_2|,
%\end{align*}
%and exchanging $z_1$ with $z_2$ we get that 
%\begin{align*}
%|f_0(x,z_1)  - f_0(x,z_2)| \leq c|z_1-z_2|.
%\end{align*}
\end{proof}

We recall the following results about finite horizon cost:
\begin{lemma}
Assume that Hypotheses \ref{hypo SDE non degenerate} and \ref{hypo controle} hold true and that $F$ is Lipschitz bounded and belong to the class $\Ga$, then for arbitrary control $a$,
\begin{align*}
J^T(x,a) \geq u(T,x),
\end{align*}
where $u(t,x)$ is the mild solution of 
\begin{align*}
\left\{ 
\begin{array}{ll}
\frac{\partial u(t,x)}{\partial t} = \LL u(t,x) + f_0(x,\nabla u(t,x)G) , & \forall (t,x) \in \R_+ \times H,\\
u(0,x) = g_0(x),&\forall x \in H.
\end{array}
\right.
\end{align*}
Furthermore, if for all $x$, $z$ the infimum is attained in \eqref{hamiltonien} then we have the equality: 
\begin{align*}
J^T(x,\overline{a}^T) = u(T,x),
\end{align*}
where $\overline{a}^T_t =  \gamma(X_t^{x,\overline{a}^T}, \nabla u(t,X_t^{x,\overline{a}^T})G)$.
\end{lemma}
\begin{proof}
See Theorem $5.3$ in \cite{FUHRMAN_TESSITORE_THE_BISMUT_ELWORTHY_FORMULA}. 
\end{proof}

Similarly, for the ergodic cost we have the following result.
\begin{lemma}
Assume that Hypotheses \ref{hypo SDE non degenerate} and \ref{hypo controle} hold true and that $F$ is Lipschitz bounded and belong to the class $\Ga$, then for arbitrary control $a$,
\begin{align*}
J(x,a) \geq \lambda,
\end{align*}
where $(v,\lambda)$ is the mild solution of 
\begin{align*}
\LL v(x) + f_0(x,\nabla v(x) G) - \lambda = 0, ~~~~\forall x \in H.
\end{align*}
Furthermore, if for all $x$, $z$ the infimum is attained in \eqref{hamiltonien} then we have the equality:
\begin{align*}
J(x,\overline{a}) = \lambda,
\end{align*}
where $\overline{a}_t = \gamma(X_t^{x,\overline{a}}, \nabla v(X_t^{x,\overline{a}})G)$.
\end{lemma}

Finally, we apply our result to obtain the following theorem.
\begin{theorem}
Assume that Hypotheses \ref{hypo SDE non degenerate} and \ref{hypo controle} hold true and that $F$ is Lipschitz bounded and belong to the class $\Ga$. For any control $a$ we have
\begin{align*}
\liminf_{T \rightarrow + \infty} \frac{J^T(x,a)}{ T} \geq \lambda. 
\end{align*}
Furthermore, if the infimum is attained in \eqref{hamiltonien} then 
\begin{align*}
J^T(x,\overline{a}^T) \underset{T \rightarrow +\infty}{\backsim} J(x,\overline{a}) T + v(x) + L.
\end{align*}
\end{theorem}
\begin{proof}
The proof is a straightforward consequence of the two previous lemmas above  and of Theorem \ref{theorem second estimate on behavior}.
\end{proof}

\section*{Acknowledgements}
The authors would like to thank the two anonymous referees for their helpful comments.

\newpage
\bibliographystyle{siam}
\bibliography{bibliographie_nouvelle.bib}
%\bibliography{/Users/pierre-yves/Desktop/Article/Bibliographie/bibliographie_nouvelle}

\end{document}